\documentclass[preprint, 12pt]{elsarticle}
\usepackage[margin=1in]{geometry}

\usepackage[parfill]{parskip}    
\usepackage{graphicx}
\usepackage{amssymb}
\usepackage{tikz}
\usepackage{tikz-cd}
\usepackage{amsmath,amsthm,amscd,amssymb}
\usepackage{latexsym}
\usepackage[colorlinks,citecolor=red,pagebackref,hypertexnames=false]{hyperref}
\hypersetup{%
  colorlinks = true,
  citecolor=red,
  linkcolor  = black
}

\usepackage{dsfont}

\usepackage{ulem}
\usepackage{soul}
\numberwithin{equation}{section}
\theoremstyle{plain}
\newtheorem{theorem}{Theorem}[section]
\newtheorem*{theorem*}{Theorem}
\newtheorem{lemma}[theorem]{Lemma}
\newtheorem{corollary}[theorem]{Corollary}

\theoremstyle{definition}
\newtheorem{definition}[theorem]{Definition}

\theoremstyle{remark}
\newtheorem{remark}[theorem]{Remark}

\newtheorem{case[theorem]}{Case}

\newcommand{\Slim}[1]{\sum\limits_{#1}}
\newcommand{\Abs}[1]{\left\lvert{#1}\right\rvert}

\allowdisplaybreaks[1]

 \usepackage{hyperref}

\date{\today}
\journal{Journal of Number Theory}
\begin{document}
\begin{frontmatter}
    \title{The Square-Root Law does Not Hold in the Presence of Zero Divisors}

    \author[me]{Nathaniel Kingsbury-Neuschotz}
    \ead{nkingsbury@gradcenter.cuny.edu}
    \affiliation[me]{organization={The CUNY Graduate Center, Mathematics},
    addressline={365 5th Ave.},
    city={New York},
    postcode={10016},
    state={NY},
    country={USA}}

    \begin{abstract}
Let $R$ be a finite ring (with identity, not necessarily commutative) and define the paraboloid $P = \{(x_1, \dots, x_d)\in R^d|x_d = x_1^2 + \dots + x_{d-1}^2\}.$ Suppose that for a sequence of finite rings of size tending to infinity, the Fourier transform of $P$ satisfies a square-root law of the form $|\hat{P}(\psi)|\leq C|R|^{-d}|P|^\frac{1}{2}$ for all nontrivial additive characters $\psi$, with $C$ some fixed constant (for instance, if $R$ is a finite field, this bound will be satisfied with $C = 1$). Then all but finitely many of the rings are fields.

Most of our argument works in greater generality: let $f$ be a polynomial with integer coefficients in $d-1$ variables, with a fixed order of variable multiplications (so that it defines a function $R^{d-1}\rightarrow R$ even when $R$ is noncommutative), and set  $V_f = \{(x_1, \dots, x_d)\in R^d|x_d = f(x_1, \dots, x_{d-1})\}$. If (for a sequence of finite rings of size tending to infinity) we have a square root law for the Fourier transform of $V_f$, then all but finitely many of the rings are fields or matrix rings of small dimension. We also describe how our techniques can establish that certain varieties do not satisfy a square root law even over finite fields.
\end{abstract}


\end{frontmatter}
 
\section{Introduction}
Exponential sums over varieties are a classical object in algebraic geometry and analytic number theory: given an affine variety $V\subseteq \mathbb{F}_q^d$, and a nontrivial additive character $\psi$ of $\mathbb{F}_q^d$, one forms the sum

$$S(V, \psi) = \sum\limits_{\vec{x}\in V}\psi(\vec{x}).$$

By Plancherel's theorem, the best level of cancellation one can hope to obtain  (uniformly over all nontrivial characters $\psi$) is
$$|S(V, \psi)|\leq C|V|^\frac{1}{2}$$

(see \cite{IoMPak}, Proposition 2.10 for a manifestation of this, or our Lemma \ref{C-bound}). After Deligne's work on the Weil conjectures (\cite{Deligne1} and \cite{Deligne2}), one may associate a sheaf to a given individual exponential sum such that the level of cancellation of the sum is controlled in terms of the ($\ell$-adic) cohomology and weight of that sheaf; in many cases one is able to obtain square-root cancellation, for instance for generalized Kloosterman sums (\cite{KiehlWeissauer} theorem V.2.1, drawing on \cite{DeligneSGA}), sums over deformations of nonsingular projective hypersurfaces (\cite{Deligne1}), and even some ``singular'' sums (\cite{Katz}). In fact, by a result of Adolphson and Sperber (\cite{AdolphSperb}) for a ``generically chosen'' polynomial in $d-1$ variables $f(X_1,\dots, X_{d-1})$, the associated exponential sum

$$S(V_f, \psi) = \sum\limits_{\vec{x}\in V_f}\psi(\vec{x})$$

satisfies the bound

$$|S(V_f, \psi)| \leq \nu(f)|V_f|^{1/2}$$

where $V_f$ is the graph of $f$ in $\mathbb{F}_q^d$ and $\nu(f)$ is a constant determined by the multidegrees of the terms of $f$ for \textit{any} nontrivial character $\psi$ of $\mathbb{F}_q^d;$ see subsection \ref{finiteFieldsStory} for details. In particular, the Fourier transform of the indicator function of this variety satisfies a square root power saving bound uniformly away from the trivial character, and thus has optimal Fourier decay.\footnote{Sets whose Fourier transforms have optimal decay are known as Salem sets (see subsection \ref{notationPrelim} for a precise definition). The term ``Salem sets'' has a long history in analysis. The term was originally used in the continuous setting for sets whose Fourier decay is optimal given their Hausdorff dimension (see for instance \cite{Bluhm} and \cite{Mattila}); the definition used in this paper is essentially that of Iosevich and Rudnev \cite{IoRud}, which has become standard in the discrete setting (as in \cite{Chen}, \cite{Fraser}, \cite{IoMorgPak}, \cite{IoMPak}, and \cite{KohShen}). In applications in analysis and geometric combinatorics, the presence of a uniform bound over \textit{all} nontrivial characters is extremely useful, as it makes arguments such as our proof of \ref{densityArg} possible. See also \cite{IoMorgPak}, \cite{IoRud}, and \cite{KohShen} for arguments of a similar flavor.}

All of these results are in line with the square root law, a useful heuristic which roughly says we should expect the following type of bound to hold:
$$\left|\sum\text{oscillating terms of modulus 1}\right|\leq C\sqrt{\text{\# of terms}}.$$
The square root law has many more manifestations throughout mathematics. For example, on a somewhat parallel track to Adolphson and Sperber's result, a generic subset (not necessarily defined by polynomials) of $(\mathbb{Z}/N\mathbb{Z})^d$ will have near-square root cancellation --- more precisely, a random subset $S$ of $(\mathbb{Z}/N\mathbb{Z})^d$ will satisfy a square-root bound up to a logarithmic factor (\cite{Babai} Theorem 5.14). For a survey of the square-root law's significance within number theory; the interested reader should see Mazur's wonderful article \cite{Mazur}.

Bounding exponential sums is typically at the heart of analysing additive configurations in sets, such as Roth's theorem (see \cite{Roth1} and \cite{Roth2}), or the finite-field Falconer distance set problem (\cite{FalconerConj}, \cite{IoRud}, \cite{KohShen}, and the references enclosed therein), and the pinned distances problem (\cite{IoMorgPak}). In particular, when a set $V$ is a Salem set, it is guaranteed to intersect the difference set of any sufficiently large set (see for instance our theorem \ref{densityArg}). Inspired by recent trends in additive combinatorics, we investigate character sums over general finite rings (with unit, not necessarily commutative), in the model case where the variety $V$ of interest is the graph of a polynomial function, and especially in the case of the parabaloid, and show that for finite rings that are not fields, one cannot get square-root cancellation. In particular, we prove the following theorem:
\begin{theorem*}
    Let $f(X_1,\dots, X_{d-1})$ be a polynomial in $\mathbb{Z}[X_1,\dots, X_{d-1}].$ Let $V_f(R)$ denote the solution set to $X_d = f(X_1,\dots, X_{d-1})$ over $R$. Suppose a sequence of finite rings  $\{R_i\}_{i = 1}^\infty$ has the property that exponential sums over $V(R_i)$ satisfy square root cancellation (for some constant). Then all but finitely many of the rings are fields or matrix rings of small dimension relative to $d$. In the case where $f(X_1,\dots, X_{d-1}) = X_1^2 + \dots + X_{d-1}^2,$ all but finitely many of the rings are fields.
\end{theorem*}
\begin{remark}
    When we work over noncommutative rings, polynomials are really sums of words, since the order of multiplications matters. When the coefficients are integers, they can all be written on the left. (\textit{Proof}: the center of a ring is closed under addition, and contains the multiplicative identity.)
\end{remark}
\begin{remark}
    We expect that the stronger result we have in the special case of the paraboloid $X_d = X_1^2 + \dots + X_{d-1}^2$ to hold much more generally; however, with our current techniques we are only able to establish our stronger result by making a direct computation of the Fourier transform of $V_f(R)$ where $R$ is a matrix ring of small dimension relative to $d$. One could in principle carry out such a computation for any variety of interest.
\end{remark}

A theorem of this type was proven by Iosevich, Murphy, and Pakianathan  \cite{IoMPak}, partly motivated by earlier work in \cite{IoRud} and \cite{CovIoPak} studying geometric configurations in modules over $\mathbb{Z}/p\mathbb{Z}$ and $\mathbb{Z}/p^\ell\mathbb{Z},$ respectively. This paper generally follows their methods, but works in much greater generality, as they only consider the solution set to the equation $xy = 1.$

Any finite ring that is not a finite field will necessarily have zerodivisors, as a finite ring with no zerodivisors is necessarily a division ring, by the pigeonhole principle, and any finite division ring is a field, by Wedderburn's theorem. The role of zerodivisors as a ``degeneracy'' of a ring has previously appeared in a paper of Tao investigating the sum-product phenomenon in general rings \cite{Tao}.

The case where the finite ring is of the form $R = \mathbb{Z}/p^m\mathbb{Z}$ is reasonably well understood, and indeed inspired \cite{IoMPak} by way of its appearance in \cite{CovIoPak}. For example, let $p$ be an odd prime, $(n, p) = 1$, and suppose $2\leq m \leq n$. Let $k = (n,p-1).$ Consider the exponential sum

$$S = \sum\limits_{x\in\mathbb{Z}/p^m\mathbb{Z}}\text{exp}\left(\frac{2\pi ix^n}{p^m}\right).$$

We may compute:
\[
\begin{split}
    S &= \sum\limits_{\substack{x\in\mathbb{Z}/p^m\mathbb{Z}\\ p|x}}\text{exp}\left(\frac{2\pi ix^n}{p^m}\right) + \sum\limits_{x\in(\mathbb{Z}/p^m\mathbb{Z})^\times}\text{exp}\left(\frac{2\pi ix^n}{p^m}\right)\\
    &= \sum\limits_{\substack{x\in\mathbb{Z}/p^m\mathbb{Z}\\ p|x}}1 + k\sum\limits_{x\in(\mathbb{Z}/p^m\mathbb{Z})^\times}\mathds{1}_{\text{n}^\text{th}\text{powers}}(x)\text{exp}\left(\frac{2\pi i x}{p^m}\right)\\
    &= p^{m-1} + \sum\limits_{x\in(\mathbb{Z}/p^m\mathbb{Z})^\times}\sum\limits_\chi \chi(x)\text{exp}\left(\frac{2\pi i x}{p^m}\right)
\end{split}
\]
where the inner summation on the last line is over all Dirichlet characters $\chi$ of order dividing $k$. These Dirichlet characters will all be induced by the Dirichlet characters of $\mathbb{Z}/p\mathbb{Z}$ of order dividing $k$, so that we have for such a character
$$\sum\limits_{x\in(\mathbb{Z}/p^m\mathbb{Z})\times}\chi(x)\text{exp}\left(\frac{2\pi i x}{p^m}\right) = \sum\limits_{x=1}^{p-1}\chi(x)\text{exp}\left(\frac{2\pi ix}{p^m}\right)\sum_{\ell = 0}^{p^{m-1}-1}\text{exp}\left(\frac{2\pi i \ell p}{p^m}\right) = 0.
$$
As a result, we simply have $S = p^{m-1},$ far worse than square root cancellation; often, multivariate exponential sums reduce to single variable sums such as the above, and would have the same decay. Our setup differs in two key ways: first, we fix the equation and vary the ring, whereas in the above the rings under consideration were constrained so $m\leq n$. In the single variable case, one may have square-root bounds over rings such as $\mathbb{Z}/p^m\mathbb{Z}$ when $m$ is large; for the quadratic Gauss sum this appears in \cite{Apostol}, chapter 8, exercise 16. Second, and importantly, we consider general finite rings, possibly noncommutative, rather than just rings of the form $\mathbb{Z}/p^m\mathbb{Z}$ (or indeed, rings of the form $\mathbb{Z}/N\mathbb{Z}$).

The key ingredient in our argument (as in the argument of \cite{IoMPak}) is a bound on the size of proper ideals of a ring whose associated $V_f(R)$ has a Fourier transform with square root decay (see Theorem \ref{idealBound}). In some circumstances, this technique even gives results over finite fields; see remark \ref{generalizedHyperbolas} for details.

\section{Proof of the Main Result}
\subsection{Notation and Preliminaries}\label{notationPrelim}
Throughout this paper, $R$ will be a finite ring (unital, with $1\neq 0$, not necessarily commutative) and $q = |R|.$  The \textit{Fourier Transform} on $R^d$ is defined as follows: given a function $f:R^d\rightarrow\mathbb{C}$, if $\psi$ is a character of the underlying additive group of $R^d,$ 

$$\widehat{f}(\psi) = \frac{1}{|R|^d}\Slim{x\in R^d}f(x)\psi(x).$$
With this definition, we have the Fourier inversion formula

$$f(x) = \Slim{\psi\in (R^d)^{\vee}}\widehat{f}(\psi)\psi(x).$$

Typically, $E$ will denote a subset of $R^d$; by abuse of notation, we will write $E$ both for the set $E$ and for its characteristic function, and similarly for sets denoted by other symbols.

A set $S\subseteq R^d$ is a $C$-Salem set if $|\widehat{S}(\psi)| \leq C\cdot |R|^{-d}\cdot |S|^\frac{1}{2}$ for all nontrivial $\psi$ in the dual group, where $\widehat{S}$ is the Fourier transform of the indicator function of $S$. 

\subsection{Square-Root Cancellation over Finite Fields}\label{finiteFieldsStory}
Over a finite field, the graphs of ``most'' polynomial functions are $C$-Salem, where the constant $C$ depends only on the degree of the polynomial. This fact essentially follows from the main result of \cite{AdolphSperb}, in which Adolphson and Sperber use a combination of Dwork's $p$-adic cohomology and an important resolution discovered by Kouchnirenko in his work on Newton polyhedra and Milnor numbers to control the shape of the $L$-function of an exponential sum associated to a polynomial and to bound the $p$-adic norms of its roots, which then allows them to control the dimensions of the $\ell$-adic cohomology groups of the associated sheaf and through an induction its weight. Before stating their key result, we recall some definitions from \cite{AdolphSperb}, which will be stated in the generality of Laurent polynomials, though we only need the special case of genuine polynomials.

\begin{definition}
Let $k$ be a field and let $f\in k[X_1, X_2, \dots, X_d, (X_1X_2\cdots X_d)^{-1}]$ be a Laurent polynomial. The \textbf{Newton polyhedron} of $f$, denoted $\Delta(f)$, is the convex hull in $\mathbb{R}^n$ of the origin $(0, 0, \dots, 0)$ and the set of multidegrees of the monomials appearing in $f$.
\end{definition}

\begin{definition}
Let $\Delta$ be a convex polyhedron in $\mathbb{R}^d$. A \textbf{face} of $\Delta$ is a set of the form

$$\{\vec{x} \in \mathbb{R}^d: \nu(\vec{x}) = \min_{\vec{y}\in \Delta}\nu(\vec(y)) \}.$$
\end{definition}

This generalizes the notion of edges and vertices of a polygon, or faces, vertices, and edges of a 3 dimensional polyhedron. No matter the dimension, a convex polyhedron has only finitely many faces. Codimension one faces are sometimes known as ``facets,'' though we will not need this term here. 

Given a Laurent polynomial $f$ with Newton polyhedron $\Delta(f)$, for any face $\sigma$ of $\Delta(f)$ we define $f_\sigma$ to be the sum of the terms of $f$ whose multidegrees lie within $\sigma$. That is to say, if we have

$$ f = \sum\limits_{\vec{j}\in \Delta(f) \cap \mathbb{Z}^d}a_{\vec{j}}\vec{X}^{\vec{j}},$$

then

$$f_\sigma = \sum\limits_{\vec{j}\in \sigma \cap \mathbb{Z}^d}a_{\vec{j}}\vec{X}^{\vec{j}},$$

where the notation $\vec{X}^{\vec{j}}$ denotes $X_1^{j_1}X_2^{j_2}X_3^{j_3}\cdots X_d^{j_d}.$ With this notation fixed, we can state the key nondegeneracy condition in \cite{AdolphSperb}.

\begin{definition}
Let $f$ be a Laurent polynomial defined over a field $k$, and let $\Delta(f)$ be its Newton polyhedron. $f$ is said to be \textbf{nondegenerate with respect to $\Delta(f)$} if for all faces $\sigma$ of $\Delta(f)$ that do not contain the origin, the polynomials
$$\frac{\partial f_\sigma}{\partial X_1}, \frac{\partial f_\sigma}{\partial X_2},\dots, \frac{\partial f_\sigma}{\partial X_d}$$
have no common zero in $(\overline{k}^\times)^d$, where $\overline{k}$ denotes an algebraic closure of $k$.
\end{definition}
Finally, we describe the constant which will appear in a Salem-type bound:
\begin{definition}
Let $[d] = \{1, 2, \dots, d\}$. For each $A\subseteq [d]$, let

$$\mathbb{R}_A^d = \{(x_1,\dots, x_d)\in \mathbb{R}^d| x_i = 0 \text{ for all } i\in A\}.$$

Let $V_A(f)$ be the volume of $\Delta(f)\cap \mathbb{R}_A^d$ with respect to the Lebesgue measure on $\mathbb{R}_A^d$. Then we define

$$\nu(f) = \sum\limits_{A\subseteq [d]} (-1)^{|A|}(d-|A|)!V_A(f).$$
\end{definition}

With these definitions in place, we can state the main result of Adolphson and Sperber in a form relevant to our discussion of Fourier transforms:

\begin{theorem}[\cite{AdolphSperb}, theorem 4.20]\label{adolphSperbBound}
Given a $d$-dimensional polyhedron $\Delta$ in $\mathbb{R}^d$ whose vertices have nonnegative integral coordinates\footnote{so that it is the Newton polyhedron of some polynomial}, there is a set $S_\Delta$ consisting of all but finitely many prime numbers such that if char$(\mathbb{F}_q)\in S_\Delta,$ then given $f\in \mathbb{F}_q[X_1, X_2,\dots, X_d]$ with $\Delta(f) = \Delta,$ if $f$ is nondegenerate and for each $i\in [d]$, $f$ contains a monomial of the form $c_iX_i^k$ with $k\geq 1$, $c_i \neq 0$, then for any nontrivial additive character $\psi$ of $\mathbb{F}_{q^n}$, we have that
$$\left\lvert\sum\limits_{\vec{x}\in \mathbb{F}_{q^n}^d}\psi(f(\vec{x}))\right\rvert \leq \nu(f)q^{\frac{nd}{2}}.$$
\end{theorem}

\begin{remark}
In \cite{AdolphSperb}, Adolphson and Sperber establish several results that are stronger than the one stated above. First, They work in the greater generality of sums over $(\mathbb{G}_m)^r\times \mathbb{A}^s$ ($r + s = d$), where $\mathbb{G}_m$ denotes the algebraic torus $k^\times$ and $\mathbb{A}$ is the affine line $k$. In this setting, the condition that for each $i\in [d]$, $f$ contains a monomial of the form $c_iX_i^k$ with $k\geq 1,$ $c_i\neq 0$ must be replaced with a somewhat more complicated condition that $f$ is \textbf{commode} or \textbf{convenient} with respect to the set of coordinates which are allowed to be zero. This condition comes from the toric decomposition of $(\mathbb{G}_m)^r\times \mathbb{A}^s.$ If $r = d$ (so $s = 0$), the condition is vacuously satisfied. 

Second, in theorem 5.17 they establish optimality of their bound in the case of sums over $\mathbb{A}^d$ in the sense that the relevant sum over $F_{q^n}$ can be expressed as a sum of $\nu(f)$ terms of modulus $q^{\frac{nd}{2}}$ (that is, it has pure weight $d$).

Last, for sums over $\mathbb{A}^d$, when the coordinates of the the exterior normal vectors to each codimension-one face $\sigma$ over $\Delta(f)$ not containing the origin are all positive, they are able to drop the restrictions on the characteristic of $\mathbb{F}_q$ (theorem 5.18). 
\end{remark}

Using their result, we may establish the ubiquity of Salemness for graphs of polynomial functions over finite fields.
\begin{theorem}
Given a $d-1$-dimensional polyhedron $\Delta$ in $\mathbb{R}^{d-1}$ whose vertices have nonnegative integral coordinates, there is a set $S_\Delta$ consisting of all but finitely many prime numbers such that if $q = p^n$ with $p \in S_\Delta,$ then given $f\in \mathbb{F}_q[X_1, X_2,\dots, X_{d-1}]$ with $\Delta(f) = \Delta,$ if $f$ is nondegenerate and for each $i\in [d-1]$, $f$ contains a monomial of the form $X_i^{k_i}$ with $k_i\geq 2$, then for any nontrivial additive character $\psi$ of $\mathbb{F}_{q}^d$, we have that
$$\widehat{V_f}(\psi) \leq \nu(f)q^{-\frac{d+1}{2}}.$$

In particular, $V_f$ is a $\nu(f)$-Salem set.
\end{theorem}
\begin{proof}
For some $a_1, \dots, a_d\in\mathbb{F}_q$, we may write 

\[\begin{split}\widehat{V_f}(\psi) &= q^{-d}\sum\limits_{(x_1, \dots, x_{d-1})\in \mathbb{F}_q^{d-1}}\text{exp}(2\pi i \text{Tr}(a_1x_1 + \dots + a_df(x_1, x_2, \dots, x_{d-1}))/p)\\
&= q^{-d}\sum\limits_{(x_1, \dots, x_{d-1})\in \mathbb{F}_q^{d-1}} \text{exp}(2\pi i \text{Tr}(F(x_1,\dots, x_{d-1}))/p)
\end{split}\]
where $\text{Tr}$ denotes the abolute trace function from $\mathbb{F}_q$ to its prime subfield $\mathbb{F}_p$ and $F(x_1,\dots, x_{d-1}) = a_{d}f(x_1, x_2, \dots, x_{d-1}) + a_1x_1 + \dots, + a_{d-1}x_{d-1}.$ If $a_d = 0$, then the sum is $0$ by orthogonality of characters. If $a_d \neq 0$, then the assumption that for each $i\in [d-1]$, $f$ contains a monomial of the form $X_i^{k_i}$ with $k_i\geq 2$ implies that $f$ and $F$ have the same Newton polyhedron $\Delta$. Furthermore, any face of $\Delta$ containing the basis vector $\vec{e_i}$ must contain the origin, as else it would separate the origin from $k_i\vec{e_i}$, which is impossible as both lie in $\Delta$. As a result, for each face not containing the origin, $F_\sigma = a_{d}f_\sigma,$ and so $F$ is nondegenerate if and only $f$ is. Thus, by theorem \ref{adolphSperbBound}, we have that

$$\sum\limits_{(x_1, \dots, x_{d-1})\in \mathbb{F}_q} \text{exp}(2\pi i \text{Tr}(F(x_1,\dots, x_{d-1}))/p) \leq \nu(f)q^{\frac{d-1}{2}}.$$
The result follows.
\end{proof}

Now, if $\Delta$ is a polyhedron whose vertices have nonnegative integral coordinates and which contains the points $k_i\vec{e_i}$ for some $k_i\geq 2$ for all $i\in [d-1]$, then with the exception of at most finitely many characteristics, there is a nonempty Zariski open subset $U$ of the moduli space of polynomials with Newton polyhedron $\Delta$ such that for any $f\in U$, we have that $f$ is nondegenerate and for each $i\in [d-1]$, $f$ contains a monomial of the form $c_iX_i^{k_i}$ with $k_i\geq 2,$ $c_i\neq 0$ (compare with the remark and conjecture after corollary 3.14 of \cite{AdolphSperb}). We therefore obtain the following corollary.

\begin{corollary}
Let $\Delta$ be a $d-1$-dimensional polyhedron in $\mathbb{R}^{d-1}$ whose vertices have nonnegative integral coordinates. Then there is a set $S_\Delta$ consisting of all but finitely many prime numbers such that if char($\mathbb{F}_q)\in S_\Delta,$ then for a generically chosen polynomial $f\in \mathbb{F}_q[X_1, X_2,\dots, X_{d-1}]$ with $\Delta(f) = \Delta,$ $V_f$ is a $\nu(f)$-Salem set in $\mathbb{F}_q^d$.
\end{corollary}

In the special case where $\Delta = \{(x_1,\dots, x_{d-1})\in\mathbb{R}^{d-1}| x_i \geq 0\text{ for all } i\text{ and } x_1 + \dots + x_{d-1}\leq n\}$ for $n\geq 2$, we may use theorem 5.18 from \cite{AdolphSperb} in place of theorem 4.20 (our theorem \ref{adolphSperbBound}) to remove the restriction on the characteristic.

\begin{corollary}
For a generically chosen polynomial $f$ of total degree at most $k$, $V_f$ is a $\nu(f)$-Salem set in $\mathbb{F}_q^d$.
\end{corollary}

\subsection{Bounding Ideals}
The key ingredient in our proof of the main result is an upper bound on the size of proper (left or right) ideals in rings over which $V_{f}(R)$ is $C$-Salem. This will be proven using a standard ``positive density'' type argument for producing structures played off a ``transfer'' lemma and the fact that a proper ideal cannot contain the identity. We use this result to prove our main theorem for finite simple rings, and then build up first to finite semisimple rings and to general finite rings, leveraging a structure theorem for finite rings.

Let $V$ be an arbitrary subset of $R^d$. Let $E\subseteq R^d.$ We define the $V$-difference set of $E$, denoted $N(E),$ by
$$N(E) = \{(x, y)\in E\times E|x-y\in V\},$$
and we set $n(E) = |N(E)|.$ If we wish to emphasize the dependence on $V$, we may write $N_V(E)$ and $n_V(E)$ in place of $N(E)$ and $n(E)$, respectively. If $V$ is $C$-Salem, then for any sufficiently large set $E$ (relative to $d$, $C$, and $R$), then $N(E)$ must be nonempty (theorem \ref{densityArg}), that is, $n(E)>0.$

\begin{theorem}\label{densityArg}
    Let $R$ be a ring, and suppose that $V\subseteq R^d$ is $C$-Salem. Suppose $E\subseteq R^d$ satisfies 
    $$\frac{Cq^d}{|V|^{\frac{1}{2}}} < |E|.$$
    Then $n(E) > 0.$
\end{theorem}
\begin{proof}
    Let $q = |R|$. We have that
   \[
   \begin{split}
      n(E) &= |\{(x, y)\in E\times E| x-y\in V\}|\\
      &= \Slim{x, y} E(x)E(y)V(x-y)\\
      &= \Slim{x, y} E(x)E(y)\Slim{\psi} \widehat{V}(\psi)\psi(x-y)\\
      &= \Slim{\psi} \widehat{V}(\psi) \left(\Slim{x} E(x)\psi(x)\right)\left(\overline{\sum\limits_y E(y)\psi(y)}\right) \\
      &= \sum\limits_\psi \widehat{V}(\psi)q^d\widehat{E}(\psi)q^d\overline{\widehat{E}(\psi)} \\
      &= q^{2d}\sum\limits_{\psi\neq \psi_0} \widehat{V}(\psi)|\widehat{E}(\psi)|^2 + q^{2d}\widehat{V}(\psi_0)|\widehat{E}(\psi_0)|^2 \\
      &= q^{2d}\sum\limits_{\psi\neq \psi_0} \widehat{V}(\psi)|\widehat{E}(\psi)|^2 + q^{2d}\frac{|V|}{q^d}\left(\frac{|E|}{q^d}\right)^2 \\
&=q^{2d}\sum\limits_{\psi\neq \psi_0} \widehat{V}(\psi)|\widehat{E}(\psi)|^2 + \frac{|V||E|^2}{q^d}.
    \end{split}
   \] 
   We let $D(E) = q^{2d}\sum\limits_{\psi\neq \psi_0} \widehat{V}(\psi)|\widehat{E}(\psi)|^2$ denote the discrepancy. We estimate it as follows:
   \[
   \begin{split}
       |D(E)| &\leq q^{2d}\text{max}_{\psi\neq \psi_0}|\widehat{V}(\psi)|\Slim{\psi\neq \psi_0}|\widehat{E}(\psi)^2|\\
       &\leq q^{2d}Cq^{-d}|V|^{\frac{1}{2}}\Slim{\psi}|\widehat{E}(\psi)^2|\\
       &= q^{d}C|V|^{\frac{1}{2}}q^{-d}|E|\\
       &= C|V|^\frac{1}{2}|E|.
   \end{split}
   \]
   $n(E)$ will be greater than 1 whenever $|D(E)|<\frac{|V||E|^2}{q^d},$ which will occur whenever
   $$C|V|^\frac{1}{2}|E|<\frac{|V||E|^2}{q^d},$$
   that is to say, whenever
   $$\frac{Cq^d}{|V|^\frac{1}{2}} < |E|,$$
   as was to be shown.
\end{proof}

\begin{remark}
    It is worth noting that \cite{IoRud} (and many others) use this style of argument to produce geometric configurations, and that \cite{CovIoPak} investigates how this argument must break down over $\mathbb{Z}/p^\ell\mathbb{Z}.$
\end{remark}
Now, suppose $V(R)$ denotes a subset of $R^d$ for each $R$ (here $d$ is fixed), where $R$ varies over the collection of all rings, such as in the case where $V = V_{f}.$ In this case we say that $R$ has $V$-Salem constant $C$ if $C$ is the smallest constant such that $V$ is a $C$-Salem set.\\

In order to apply this to $V_f$, we will need to count $R$-points. For much of our argument, we will actually need to work with the family of varieties given by translations in the last variable $X_d.$ To simplify notation, we assume that $f$ has constant term $0$. For each $c\in\mathbb{Z}$, we write $V_{f, c} = \{(x_1, x_2, \dots, x_d): x_d = f(x_1, \dots, x_{d-1}) + c\},$ the graph of the function $f+c$.  Let's now count $R$-points, and describe how these varieties relate to one another.
\begin{lemma}\label{parabolaPointCount}
    $|V_{f, c}(R)| = q^{d-1}$. 
\end{lemma}
\begin{proof}
    For each choice of values for $x_1,\dots x_{d-1}$, we are given exactly one value of $x_d$.
\end{proof}
\begin{lemma}\label{allParabolasAreTheSame}
    $V_{f, c}$ is $C$-Salem if and only if $V_{f, 1}$ is.
\end{lemma}
\begin{proof}
    We have that \[
    \begin{split}
        \widehat{V_{f, c}}(\psi) &= \frac{1}{q^d}\Slim{x} V_{f, c}(x)\psi(x)\\
        &= \frac{1}{q^d}\Slim{x} \psi(x_1, x_2,\dots, x_{d-1}, f(x_1, x_2, \dots, x_{d-1}) + c)\\
        &= \frac{\psi(0,0,\dots,0,c-1)}{q^d}\Slim{x} \psi(x_1, x_2, \dots x_{d-1}, f(x_1, x_2, \dots, x_{d-1}) + 1)\\
        &= \psi(0,0,\dots,0,c-1)\widehat{V_{f, 1}}(x).\\
    \end{split}\]
    The two quantities have the same modulus, so one satisfies the $C$-Salem bound if and only if the other does.
\end{proof}
With this ``transfer'' lemma accomplished, we can now prove the key bound:
\begin{theorem}\label{idealBound}
    Let $R$ be a ring, with $|R| = q$, and suppose (for any $c$) that the $(d-1)$-dimensional variety $V_{f, c}\subseteq R^d$ is $C$-Salem. Let $I$ be a proper left (or right) ideal. Then 
    $$|I| \leq C^\frac{1}{d}q^{\frac{1}{2} + \frac{1}{2d}}.$$
\end{theorem}
\begin{proof}
    By the previous lemma, we may assume the given bound is of the form ``$V_{f,1}$ is $C$-Salem''. We also recall that $|V_{f, 1}| = q^{d-1}.$ Let $I$ be a left or right ideal of $R$, and let $E = I^d.$ If $x, y\in E$ with $x-y \in V_{f, 1}$, then 
    $$f(x_1-y_1, \dots, x_{d-1} - y_{d-1}) - (x_d - y_d) = 1,$$ contradicting properness of $I$, since $f$ is a polynomial with $0$ constant term, so $n(E) = 0.$ Thus, $|E| = |I|^d \leq \frac{Cq^d}{|V_{f,1}|^{\frac{1}{2}}} = Cq^\frac{d+1}{2},$ whence $|I| \leq C^\frac{1}{d}q^{\frac{1}{2} + \frac{1}{2d}},$ as desired.
\end{proof}
Now that we can bound the size of left and right ideals in $R$, we are off to the races. To prove our theorem, it suffices to show that only finitely many (isomorphism classes of) non-field finite rings have $V_{f, c}$-Salem number less than or equal to a fixed constant $C$. We also note that finite rings have a simple structure: every finite ring has a two-sided ideal $J$, called the Jacobson ideal, such that the quotient ring $R/J$ is a (finite) direct product of matrix rings $M_{n\times n}(F)$ where $F$ is a finite field (not that the case $n=1$ encompasses the case where a factor is a field). These matrix rings are exactly the so-called \textit{simple} finite rings. A proof of this may be found in section 2.4 of \cite{IoMPak}. We exploit this structure: we first prove our theorem for simple finite rings, then for their products (the so-called \textit{semisimple} rings), and then finally for general rings.
\subsection{Simple Rings}
\begin{theorem}\label{mostCases}
    Let $R$ be a finite simple ring (that is, the ring of $n\times n$ matrices over a finite field $F$). Suppose that  $V_{f, c}$ is a $C$-Salem set. If $d \geq 4$, then $R$ is either a finite field, the ring of $2\times 2$ matrices over a field, or comes from a finite list of exceptional rings. If $d = 3$, then in additional to those possibilities $R$ could be a ring of $3\times 3$ matrices over a finite field, and if $d = 2$ then in addition to those possibilities $R$ could be a $3\times 3$ or $4\times 4$ ring of matrices over a finite field.
\end{theorem}
To simplify the statement of this result, we introduce the following definition:
\begin{definition}
    We say that a matrix ring $M_{n\times n}(F)$ has \textbf{small dimension relative to $d$} if either $n = 2,$ $n = 3$ and $d = 2$ or $3,$ or $n = 4$ and $d = 2.$
\end{definition}
With this definition, we may restate theorem \ref{mostCases} as
\begin{theorem*}
    Let $R$ be a finite simple ring (so it is the ring of $n\times n$ matrices over a finite field $F$). Suppose that  $V_{f, c}$ is a $C$-Salem set. Then $R$ is either a finite field, a matrix ring of small dimension relative to $d$, or comes from a finite list of exceptional rings (the list depends on C).
\end{theorem*}
\begin{proof}
    Let $I$ be the maximal ideal of all matrices with right column 0. Clearly $|R| = |F|^{n^2}$ and $|I| = |F|^{n^2-n}$. Applying the previous lemma, we have that:
    $$|F|^{n^2-n}\leq C^\frac{1}{d}|F|^{\frac{n^2}{2} + \frac{n^2}{2d}}.$$
    Rearranging, we have:
    $$|F|^{\frac{n^2}{2}-\frac{n^2}{2d} - n }\leq C^\frac{1}{d}.$$
    Clearly, since $|F|\geq 2,$ this inequality may only hold for finitely many choices of $n.$ For a fixed value of $n$, if $\frac{n^2}{2}-\frac{n^2}{2d} - n > 0,$ this may only hold for finitely many values of $|F|,$ and thus for finitely many (isomorphism classes of) finite fields. This occurs if $n^2(\frac{1}{2} - \frac{1}{2d}) > n,$ i.e., if $\frac{d-1}{2d} > \frac{1}{n},$ which may be re-arranged as $1-\frac{1}{d} > \frac{2}{n},$ or as $n<\frac{2}{1-\frac{1}{d}}.$
    
    If $d \geq 4,$ $1-\frac{1}{d} \geq \frac{3}{4}$, so the inequality holds if $\frac{3}{4} > \frac{2}{n},$ i.e. if $n < \frac{8}{3} < 3,$ so either we have a finite field ($n = 1$) or a ring of $2\times 2$ matrices. We remark that there is no way from this argument alone to rule out the $2\times 2$ matrices, just as in proposition 2.7 of \cite{IoMPak}, even by choosing a really large dimension, as we will always have an inequality of the form $n < \frac{2}{1-\epsilon}.$

    If $d = 3,$ we may only rule out the cases where $n$ satisfies $\frac{2}{3} > \frac{2}{n},$ so can only rule out $n\times n$ matrices for $n>3$; for $d = 2$ we get $\frac{1}{2} > \frac{2}{n},$ and can only rule out $n> 4.$ 
\end{proof}

\begin{remark}\label{generalizedHyperbolas}

Using these same techniques, we can also very quickly show that certain varieties do \textbf{not} have optimal Fourier decay, even over finite fields. As an example, we will work with the generalized hyperbola. Fix a dimension $d\geq 3$, and define the generalized hyperbola $H = \{x = (x_1, x_2, \dots x_d)\in R^d | x_1x_2\dots x_d = 1\}$. This variety is also known as the \textit{Hamming Variety}.

By a result of Deligne, for a frequency vector $(a_1,\dots, a_d)$ with no nonzero components in $\mathbb{F}_q^d$ $(q = p^k)$, we have that

$$\left\lvert\sum\limits_{x_1x_2\cdots x_d = 1}exp(2\pi i Tr(a_1x_1 + \dots + a_dx_d)/p)\right\rvert \leq C(d) q^{\frac{d-1}{2}},$$

in line with the square root law heuristic. However, when the frequence vector has some zero components, the decay is not as good. The precise details were worked out by Cheong, Koh, and Pham in \cite{Hamming}:

\begin{theorem}[\cite{Hamming}, lemma 3.1]
Suppose that $\ell$ of the components of $(a_1, a_2, \dots, a_d)$ are 0. If $\ell = 0$, then we have

$$\left\lvert\sum\limits_{x_1x_2\cdots x_d = 1}exp(2\pi i Tr(a_1x_1 + \dots + a_dx_d)/p)\right\rvert \leq C(d) q^{\frac{d-1}{2}};$$

if $1\leq \ell \leq d,$ then we have

$$\left\lvert\sum\limits_{x_1x_2\cdots x_d = 1}exp(2\pi i Tr(a_1x_1 + \dots + a_dx_d)/p)\right\rvert = (q-1)^{\ell - 1},$$

In particular, if at least one of the components of the frequency vector $(a_1,\dots, a_d)$ is zero, then $|\widehat{H_j}(a_1,\dots, a_d)|$ is of order $q^{-(d+1-\ell)}$.
\end{theorem}

Interestingly, the relatively qualitative techniques of our paper can be used to show the weaker result that the $(d-1)$-dimensional generalized hyperboloid is not a $C$-Salem set over finite fields assuming that $d\geq 4$. To avoid complications having to do with $j$ having common factors with the characteristic of the underlying ring, we concentrate on $H_1$. 

First, we count points. Each coordinate of $x\in H_1$ must be a unit, and given a choice of $d-1$ units $x_1, x_2, \dots, x_{d-1},$ there is exactly one choice of $x_d$ such that $x_1x_2x_3\dots x_{d-1}x_d = 1,$ namely $x_d = x_{d-1}^{-1}x_{d-2}^{-1}\cdots x_1^{-1},$ so $|H_1| = |R^*|^{d-1}.$

Thus, in our case, theorem \ref{densityArg} says that:

\begin{theorem}
    Let $R$ be a ring, and suppose that $H_1\subseteq R^d$ is $C$-Salem. Suppose $E\subseteq R^d$ satisfies 
    $$\frac{Cq^d}{|R^*|^{\frac{d-1}{2}}} < |E|.$$
    Then $n(E) > 0.$
\end{theorem}

(Recall that $N(E) = \{(x, y)\in E\times E|x-y\in H_1,\}$ and $n(E) = |N(E)|.$ )

For the generalized hyperbola, we can get a much stronger bound on the size of ideals, because we only need one coordinate to live inside the ideal, not all of them.

\begin{theorem}
    Let $R$ be a ring, and suppose that the $d$-dimensional hyperbola $H_1\subseteq R^d$ is $C$-Salem. Let $I$ be a proper left (or right) ideal of $R$. Then
    $$|I|\leq\frac{Cq}{|R^*|^\frac{d-1}{2}}.$$
\end{theorem}
\begin{proof}
    We prove the result for left ideals $I$; right ideals follow similarly. Let 
    $$E = \underbrace{R\times R\times \cdots R}_{d-1 \,times}\times I.$$
    Clearly, for $x = (x_1, x_2, \dots x_d), y = (y_1, y_2, \dots, y_d)\in E,$ $(x_d-y_d)\in I,$ so $(x_1-y_1)(x_2-y_2)\dots (x_{d-1}-y_{d-1})(x_d - y_x)\in I.$ As $I$ is proper, $1\notin I$, so $n(E) = 0.$ By the previous result, $E = q^{d-1}|I|\leq \frac{Cq^d}{|R^*|^\frac{d-1}{2}},$ so 
    $$|I| \leq\frac{Cq}{|R^*|^\frac{d-1}{2}},$$
    as was to be shown.
     
\end{proof}

For matrix rings, we expect the ratio $q/|R^*|$ to tend to $1$ as $q$ gets large, so we should be able to achieve $|I|<1$ for any proper ideal of any sufficiently large matrix ring satisfying optimal Fourier decay so long as $d > 3$. This is impossible, as $\{0\}$ is a nonempty proper ideal of any ring, so that $H_1$ is only $C$-Salem for finitely many simple rings. We now make this precise:
\begin{theorem}
    Let $d>3$, and fix $C>0.$ $H$ is only $C$-Salem over finitely many simple finite rings, including fields.
\end{theorem}
\begin{proof}
    Let $\phi_R(n, |F|)$ be as in \cite{IoMPak}, recalling the bound $\phi_R(n, |F|) \geq \frac{1}{4}.$ We estimate the size of the left ideal $I$ of matrices with right column $0$. Applying the previous theorem, we have that:

        $$|F|^{n^2-n}\leq\frac{C|F|^{n^2}}{(|F|^{n^2}\phi_R(n, |F|))^\frac{d-1}{2}},$$
        so (squaring both sides)
        $$|F|^{2n^2-2n} \leq \frac{C^2|F|^{2n^2}}{(|F|^{n^2}\phi_R(n, |F|))^{d-1}},$$
        and re-arranging and applying the lower bound on $\phi_R$,
        $$|F|^{dn^2-n^2-2n}\left(\frac{1}{4}\right)^{d-1}\leq C^2,$$
        so 
        $$|F|\leq \left(4^{d-1}C^2\right)^\frac{1}{dn^2-n^2-2n} = \left(4^{d-1}C^2\right)^\frac{1}{n(dn-n-2)}.$$
        As $d> 3$, we have that $d-1\geq 3,$ and $(d-1)n-2 \geq 1,$ for all $n$, so for any fixed $C$ and $n$, only finitely many choices of $F$ satisfy the bound. Further, as $|F|\geq 2,$ there are only finitely many choices of $n$ for which this holds (for any $|F|$), and thus there are only finitely many matrix rings for which $H$ is $C$-Salem.
\end{proof}

The generalized hyperbola only fails the Salem condition by a hair, having square root decay for ``most'' frequency vectors. This example suggests an interesting question for future research, kindly pointed out by the anonymous reviewer: if one weakens the $C$-Salem condition by allowing a few frequency vectors to have weaker decay, what can be said about the structure of the underlying ring?
\end{remark}

\subsection{Semi-Simple Rings}
Now, we extend these results to semi-simple rings. To do so, we will need two lemmas. The first is a formula for the sharpest possible Salem-constant in a product of finite rings:
\begin{lemma}\label{productsAreNice}
    Let $C_R$ denote the $V_{f, c}$-Salem number of $R$ (note that by lemma \ref{allParabolasAreTheSame} this is well-defined and independent of $c$). Suppose $R = R_1\times R_2$ is a nontrivial product decomposition. Then $C_R = \text{max}(C_{R_1}|V_{f, c}(R_2)|, C_{R_2}|V_{f, c}(R_1)|).$
\end{lemma}
\begin{proof}
    As $R = R_1\times R_2$ is also a decomposition of the underlying abelian groups, the (irreducible) characters of $R$ are products of characters of $R_1$ and $R_2.$ The variety $V_{f, c}$ also decomposes as $V_{f, c}(R) = V_{f, c}(R_1)\times V_{f, c}(R_2)$ under the decomposition $R^d = R_1^d\times R_2^d.$ For any $\psi = (\psi_1, \psi_2)\in (R_1^d)^\vee\times (R_2^d)^\vee = (R^d)^\vee$ we may see
    \[
    \begin{split}
        \widehat{V_{f, c}(R)}(\psi_1, \psi_2) &= \frac{1}{|R|^d}\Slim{(\overrightarrow{x_1}, \overrightarrow{x_2})\in V_{f,c}(R_1)\times V_{f,c}(R_2)}\psi_1(-\overrightarrow{x_1})\psi_2(-\overrightarrow{x_2})\\
        &= \widehat{V_{f, c}(R_1)}(\psi_1)\widehat{V_{f, c}(R_2)}(\psi_2).
    \end{split}
    \]
    The maximum of $|\widehat{V_{f, c}(R_i)}(\psi_i)|$ as $\psi_i$ varies over the nonzero elements is by definition $C_{R_i}|R_i|^{-d}|V_{f, c}(R_i)|^\frac{1}{2}$, while the value of $|\widehat{V_{f, c}(R_i)}(\psi_0)|$ is $|R_j|^{-d}||V_{f, c}(R_i)|$ for $i = 1,2$ by the definition of the Fourier transform. Thus we have that
    $$\cfrac{|R|^d}{\sqrt{|V_{f, c}(R)}}\max_{(m_1, m_2)\neq(0, 0)}|\widehat{V_{f, c}(R)}(m_1, m_2)| = \max\{C_{R_1}C_{R_2}, C_{R_1}|V_{f, c}(R_2)|^\frac{1}{2}, C_{R_2}|V_{f, c}(R_1)|^\frac{1}{2}\}  
$$
    
    by considering the three cases where $(m_1, m_2)$ are both nonzero, $m_1 = 0$, and $m_2 = 0.$ The trivial bound $C_{R_j}\leq\sqrt{|V_{f, c}(R_j)|}$ shows that the maximum is one of the last two terms.
\end{proof}
The latter of our lemmas is a lower bound on the $V_{f, c}$-Salem constant of any ring. For our purposes, a rather weak bound will do. As the bound depends only on point-counting and Fourier analysis, and not on any algebraic structure, we state it for a general set $E$:
\begin{lemma}\label{C-bound}
    Let $E\subseteq R^d,$ and suppose $E$ is $C$-Salem. Then
    $$\sqrt{1-\frac{|E|}{|R|^d}} \leq C.$$
    In particular, if $V_{f, c}$ is $C$-Salem, then $C\geq \sqrt\frac{1}{2}$
\end{lemma}
\begin{proof}
   By Plancherel's theorem, we have
    $$|R|^d\Slim{m\in R^d}|\widehat{E}|^2 = \Slim{x\in R^2}|E(x)| = |E|.$$
    Thus
    $$|R|^d\left(\frac{|E|^2}{|R|^{2d}} + \Slim{m\neq 0}|\widehat{E}|^2\right) = |E|,$$
    so
    $$\frac{|E|}{|R|^d} - \frac{|E|^2}{|R|^{2d}} \leq \Slim{m \neq 0}C^2|R|^{-2d}|E|\leq C^2|R|^{-d}|E|,$$
    and hence
    $$1-\frac{|E|}{|R|^d}\leq C^2$$
    from which the first result follows. When $E$ = $V_{f, c},$ we have by lemma \ref{parabolaPointCount} that $|V_{f, c}| = |R|^{d-1},$ whence the lower bound becomes $\sqrt{1-\frac{1}{|R|}},$ and as $|R| \geq 2$ (as $1\neq 0$ in our rings) the second bound follows. 
\end{proof}
Now, we have the following theorem:
\begin{theorem}\label{semisimpleRings}
    Let $R$ be a finite semi-simple ring, and suppose that $V_{f, c}$ is a $C$-Salem set. Then $R$ is either a finite field or a matrix ring of small dimension relative to $d$, or comes from a finite list of exceptional rings (the list depends on $C$).
\end{theorem}
\begin{proof}
    Suppose that $R$ is semisimple. If $R$ is in fact simple, then this is just Theorem \ref{mostCases}. If not, then suppose we may write $R = M_{n_1\times n_1}(F_1)\times R_2,$ where $R_2$ is nontrivial and semisimple. Let $C_R$ be the $V_{f, c}$-Salem constant of $R$, and define $C_{R_1}$ and $C_{R_2}$ similarly. By lemma \ref{productsAreNice}, $C_{R_2}|V_{f, c}(M_{n_1\times n_1}(F_1))|^\frac{1}{2}\leq C_R,$ and $C_{R_1}|V_{f, c}(R_2)|^\frac{1}{2}\leq C_R.$ Using our lower bounds on $C_1$ and $C_2,$ we get upper bounds on $|V_{f, c}(M_{n_1\times n_1}(F_1))|$ and on $|R_2|$:
    $$|V_{f, c}(M_{n_1}(F_1))| = |M_{n_1\times n_1}(F_1)|^{(d-1)}\leq 2C_R^2,$$
    $$|V_{f, c}(R_2)| = |R_2|^{d-1}\leq 2C_R^2.$$
    In particular, we have an upper bound on the size of $|R_2|$, and thus $R_2$ can only come from a finite list of rings. As for $M_{n_1\times n_1}(F_1),$ we have that $|M_{n_1\times n_1}(F_1)| = |F_1|^{n_1^2},$ so we have an upper bound on $|F_1|$ and, as $|F_1|\geq 2$, an upper bound on $n_1$ as well. Thus, there are only finitely many possibilities for $M_{n_1\times n_1}(F_1)$, and hence only finitely many (non-simple) possibilities for $R$. 
\end{proof}
\begin{remark}
    Due to the simiplicity of the point-counting for $V_{f, c}(R),$ this proof is much simpler than the corresponding proof for hyperbolas, as we can directly bound the size of each factor of $R$, rather than just the unit group of each factor of $R$.
\end{remark}
\subsection{General Rings}
Now, we need only show that at most finitely many rings with a bounded $V_{f, c}$-Salem number have a non-zero Jacobson radical. First, we show that a ring with a non-zero Jacobson radical has larger $V_{f, c}$-Salem number than its semi-simple part.
\begin{lemma}\label{JacobsonBound}
    Let $R$ be a finite ring with Jacobson radical $J$, and let $S = R/J$ be the semisimple part of $R$. If $C_R$ and $C_S=C_{R/J}$ are the $V_{f, c}$-Salem numbers of $R$ and $S$, then
    $$C_R\geq C_S|J|^\frac{d-1}{2}.$$
\end{lemma}
\begin{proof}
    We follow the strategy of section 2.8 of \cite{IoMPak}. Let $\psi,$ be a nontrivial additive character of $S^{d}$. We may pull it back under the quotient map $\pi: R\rightarrow S$ to an additive character of $R$ which is equal to $1$ on $J^{d}$ (by abuse of notation we identify the resulting character of $R^d$ with that of $S^d$). For this character, we have that:
    \[\begin{split}
        \widehat{V_{f, c}(R)}(\psi) &= \frac{1}{|R|^d}\Slim{\overrightarrow{x}\in V_{f,c}(R)}\psi(\overrightarrow{x})\\
        &= \frac{1}{|R|^d}\Slim{(x_1, \dots, x_{d-1})\in R^{d-1}}\psi(x_1, x_2\dots x_{d-1}, f(x_1, x_2, \dots, x_{d-1}) + c)\\
        &=\frac{|J|^{d-1}}{|R|^d}\Slim{(x_1, \dots, x_{d-1})\in S^{d-1}}\psi(x_1, x_2\dots x_{d-1}, f(x_1, x_2, \dots, x_{d-1}) + c)\\
        &= \frac{|J|^{d-1}}{|R|^d}\Slim{\overrightarrow{x}\in V_{f,c}(S)}\psi(\overrightarrow{x}),
    \end{split}
    \]
    since the character $\psi$ only depends on the image of $\overrightarrow{x}$ in $S^d$, and since each $(x_1, x_2, \dots, x_d)\in V_{f, c}(S)$ has exactly $|J|^{d-1}$ pre-images in $V_{f, c}(R)$, as we may freely choose which pre-image to take for the first $d-1$ coordinates, and are then left with only one possible choice satisfying the equation for the last coordinate. Taking the maximum over the nontrivial characters of $S^d,$ we get that
    $$\text{max}_{m\in (S^d)^\vee}|\widehat{V_{f,c}(R)}(m)| = \frac{|J|^{d-1}}{|R|^d}C_S|{V_{f,c}(S)}|^\frac{1}{2},$$
    whence
    $$C_R|R|^{-d}|V_{f,c}(R)|^\frac{1}{2} \geq \frac{|J|^{d-1}}{|R|^d}C_S|V_{f,c}(S)|^\frac{1}{2}.$$
    Using $|V_{f,c}(R)| = |R|^{d-1} = |J|^{d-1}|S|^{d-1} = |J|^{d-1}|V_{f,c}(S)|,$ this simplifies to
    $$C_R \geq C_S|J|^{d-1}|J|^{-\frac{d-1}{2}} = C_S|J|^\frac{d-1}{2}$$
    as was to be shown.
\end{proof}
Now, we can finally prove the general case.
\begin{theorem}\label{generalCaseParabola}
    Let $R$ be a finite ring with $V_{f, c}$-Salem number $C$. Then either $R$ is a field, a matrix ring of small dimension relative to $d$, or comes from a finite list of exceptional rings (the list depends on $C$). 
\end{theorem}
\begin{proof}
    For $J = 0,$ this is theorem \ref{semisimpleRings}. Suppose then that $J\neq 0$. From lemma \ref{C-bound} and lemma \ref{JacobsonBound}, $|J|\leq (2C^2)^\frac{1}{d-1},$ and $C_{R/J}\leq C.$ As $R/J$ is semisimple, theorem \ref{semisimpleRings} shows that save for in a finite number of cases, $R/J$ is a finite field or a small matrix ring. As $J$ is bounded, this leads to a finite list of possible ring sizes where $R/J$ is \textbf{not} a finite field or a matrix ring of small dimension relative to $d$, and as a finite set only has finitely many rings structures, this gives only finitely many possible rings. Thus, we must merely show that if we have the given Salem bound and an exact sequence
    $$0\rightarrow J\rightarrow R\rightarrow M_{n\times n}(F) \rightarrow 0,$$
    where $F$ is a finite field and $n$ is fixed, then $J$ must be 0 in all but finitely many cases.\\
    If $J\neq 0$, then $J/J^2\neq 0$ by Nakayama's lemma (\cite{Rotman}, Corollary C-2.8). $J/J^2$ is an $R/J = M_{n\times n}(F)$-module, and (composing the $R$-action with the ring morphism $F\rightarrow M_{n\times n}(F)$ by $x\mapsto xI$ where $I$ is the $n\times n$ identity matrix) an $F$-vector space. Thus $|F|$ divides $|J|$. As $|J|$ is bounded, so is $|F|$. As $n$ is fixed, this means there are only finitely many possible choices of $R/J.$
\end{proof}

\section{Parabaloids}
In the case of the paraboloid $X_d = X_1^2 + \dots + X_{d-1}^2 + c$, we can eliminate the matrix rings by direct computations.  Let $p(x_1\dots, x_{d-1}) = x_1^2 + \dots + x_{d-1}^2.$ We start for $d = 2;$ it will be seen that this family of cases is where all the essential work really is. We start with a lemma about some relevant character sums. Recall that any irreducible character of a finite field $F$ having characteristic $p$ is a map of the form $$\psi_b: x\mapsto \exp(2\pi i Tr(bx)/p),$$ where $Tr$ is the absolute trace map $F\rightarrow F_p$ and $b\in F$ (see for instance \cite{LidlNied}, Chapter 5, Theorem 5.7).
\begin{lemma}\label{gaussSums}
    Let $F$ be a finite field of characteristic $\neq$ 2. Then the sum
    $$\Slim{a\in F}\psi(-a-a^2)$$
    has modulus $|F|^{1/2},$ where $\psi$ is a nontrivial irreducible character of $F$. If $F$ has characteristic 2, and we choose $\psi = \psi_1$, then the modulus of this sum is simply $|F|.$
\end{lemma}
\begin{proof}
    Suppose $F$ is a field of characteristic not equal to 2. Then we may reduce to the standard quadratic Gauss sum by completing the square: let $x$ be the multiplicative inverse of 2. Then (replacing $a$ with $a-x$ in the sum) we have
    \[
    \begin{split}
        \Slim{a}\psi(-a^2 - a) &= \Slim{a}\psi(-(a-x)^2 - (a-x))\\
        &= \Slim{a}\psi(-a^2 + 2ax - x^2 - a - x)\\
        &= \Slim{a}\psi(-a^2 + a - x^2 - a - x)\\
        &= \Slim{a}\psi(-a^2)\psi(-x^2 - x) \\
        &= \psi(-x^2 - x)\Slim{a}\psi(-a^2)
    \end{split}
    \]
    As $F$ is finite, characters have modulus 1, so $\Abs{\psi(-x^2 - x)\Slim{a}\psi(-a^2)} = \Abs{\Slim{a}\psi(-a^2)},$ and we need only worry about (a close cousin of) the classical quadratic Gauss sum. Now,
    \[
    \begin{split}
        \Abs{\Slim{a}\psi(-a^2)}^2 &= \Slim{a, b}\psi(-a^2)\overline{\psi(-b^2)}\\
        &= \Slim{a, b}\psi(b^2 - a^2)\\
        &= \Slim{a, b}\psi((b-a)(b+a))\\
        &= \Slim{u, v}\psi(uv)\\
        &= \Slim{u}\Slim{v}\psi(uv)
    \end{split}
    \]
    where we have made the substitution $u = b-a,$ $v = b+a.$ Now, in the inner sum, if $u\neq 0$, the sum is 0 by Schur orthogonality, and if $u = 0$, each term is simply 1, so the sum evaluates to $|F|.$ Thus, $\Abs{\Slim{a}\psi(-a^2)}^2 = |F|,$ from which the result follows.\\
    
    Now, suppose the characteristic of the field is equal to 2. We may no longer complete the square as before, but now the map $a\mapsto a^2$ is an automorphism fixing the prime field, so by the additive form of Hilbert's theorem 90 (see \cite{Lang}, VI, Theorem 6.3), coupled with the fact that the characteristic is 2, $Tr(-a^2 - a)= Tr(a- a^2) = 0,$ and so $\psi_1(-a-a^2) = 1$ for all $a$, from which the result follows 
\end{proof}
Now, we begin estimating Fourier coefficients for various small $n$ (at first  with $d=2$) to show that these annoying remaining cases behave as expected.
\begin{lemma}\label{lowDimensionalExceptions}
    Fix $C$. Let $d = 2$. Then the set of $2\times 2$, $3\times 3$, and $4\times 4$ matrix rings for which $V_{p,0}$ (equivalently, $V_{p, c}$ for any $c$) is a $C$-Salem set is finite.
\end{lemma}
\begin{proof}
    We start with $2\times 2$ matrix rings, that is, with $R = M_{2\times 2}(F).$ We identify $R$ with its Pontryagin dual in such a way that the characters of R are written as $\psi_1(Tr(A\cdot-))$ for various choices of $A$. We compute the Fourier transform of $V_{p, 0}$ evaluated at $A = \begin{pmatrix} 1&0\\0&0\end{pmatrix}$:
\begin{equation}
\begin{split}
\frac{|R|^2}{|V_{p,0}|^\frac{1}{2}}\Abs{\widehat{V_{p,0}}(A, A)} &= \frac{1}{|F|^2}\Abs{\Slim{C\in M_{2\times 2}(F)} \psi_1(Tr(-AC-AC^2))} \\
&= \frac{1}{|F|^2}\Abs{\Slim{a, b, c, d} \psi_1\left(Tr\left(-A\begin{pmatrix} a&b\\c&d\end{pmatrix} - A\begin{pmatrix} a&b\\c&d\end{pmatrix}^2\right)\right)}\\
&=\frac{1}{|F|^2}\Abs{\Slim{a, b, c, d}\psi_1(-a-a^2-bc)}\\
&=\frac{1}{|F|^2}|F|\Abs{\Slim{a}\psi_1(-a-a^2)}\Abs{\Slim{b} \Slim{c}\psi_1(-bc)}\\
&=  \Abs{\Slim{a}\psi_1(-a-a^2)}
\end{split}
\end{equation}

where we used the same trick for computing $\Slim{b}\Slim{c}\psi_1(-bc)$ as in the proof of lemma \ref{gaussSums}. By lemma \ref{gaussSums}, this quantity grows as $|F|$ grows, so only finitely many $2\times 2$ matrix rings can achieve any fixed Salem-type bound.
Next we tackle $3\times 3$ matrices. We let $A = \begin{pmatrix} 1&0&0\\0&0&0\\0&0&0\end{pmatrix}$, write 
$$C = \begin{pmatrix}a&b&c\\d&e&f\\g&h&i\end{pmatrix},$$
$$AC^2 = \begin{pmatrix}a^2 + bd + cg& 0 & 0\\0 & 0 & 0\\0 & 0 & 0\end{pmatrix},$$
and compute:
\[
\begin{split}
\frac{|R|^2}{|V_{p, 0}|^\frac{1}{2}}\Abs{\widehat{V_{p, 0}}(A, A)} &= \frac{1}{|F|^\frac{9}{2}}\Abs{\Slim{C\in M_{3\times 3}(F)} \psi_1(Tr(-AC-AC^2)) }\\
&= \frac{1}{|F|^\frac{9}{2}}\Abs{\Slim{e, f, h, i}\Slim{a, b, c, d, g}\psi_1(-a-a^2- bd - cg)}\\
&= \frac{|F|^4}{|F|^\frac{9}{2}}\Abs{\Slim{a}\psi_1(-a-a^2)}\Abs{\Slim{b, d}\psi_1(-bd)}\Abs{\Slim{c, g}\psi_1(-cg)}\\
&= \frac{|F|^6}{|F|^\frac{9}{2}} \Abs{\Slim{a}\psi_1(-a-a^2)}\\
&= |F|^\frac{3}{2} \Slim{a}\psi_1(-a-a^2),
\end{split}
\]
which grows as $|F|$ grows. 
Finally, we handle the case of $n = 4$. Let
$$A = \begin{pmatrix}1&0&0&0\\0&0&0&0\\0&0&0&0\\0&0&0&0\end{pmatrix},\text{ and}$$
$$C = \begin{pmatrix}a&b&c&d\\e&f&g&h\\i&j&k&l\\m&n&o&p\end{pmatrix}.$$
Then $Tr(-AC-AC^2)=-a-a^2-be-ci-dm,$ and so the relevant computation becomes:
\[
\begin{split}
\frac{|R|^2}{|V_{p, 0}|^\frac{1}{2}}\Abs{\widehat{V_{p, 0}}(A, A)} &= \frac{1}{|F|^8}\Abs{\Slim{C\in M_{4\times 4}(F)} \psi_1(Tr(-AC-AC^2))} \\
&= \frac{|F|^9}{|F|^8}\Abs{\Slim{b, e}\psi_1(-be)}\Abs{\Slim{c, i}\psi_1(-ci)}\Abs{\Slim{d, m}\psi_1(-dm)}\Abs{\Slim{a}\psi_1(-a-a^2)}\\
&= |F|^4\Abs{\Slim{a}\psi_1(-a-a^2)}.
\end{split}
\] 
Which most definitely grows as $|F|$ grows. Having now handled all the cases, we have proven the lemma.
\end{proof}
At this point, we now know that the standard parabola is $C$-Salem for only finitely many non-field simple rings. Now we address the situation when $d>2.$
\begin{lemma}\label{higherDimensionalExceptions}
Let $d > 2$. The $d$-dimensional paraboloid is $C$-Salem for only finitely many non-field simple rings.
\end{lemma}
\begin{proof}
Fix $d>2$. By Theorem \ref{mostCases}, we only need to consider $n\times n$ matrix rings for $n = 2, 3$. First we consider the case of $n = 2$. Let $A$ be as in the proof of lemma \ref{lowDimensionalExceptions} for $n = 2$. We have that:
\[
\begin{split}
\frac{|R|^d}{|V_{p, 0}|^\frac{1}{2}}\Abs{\widehat{V_{p, 0}}(A, A,\dots, A)} &= \frac{1}{|F|^{\frac{d-1}{2}n^2}}\Abs{\Slim{C_1\dots C_{d-1}\in M_{2\times 2}(F)} \psi_1(Tr(-\sum\limits_{i = 1}^{d-1}(AC_i+AC_i^2)))} \\
&= \frac{1}{|F|^{2d-2}}\prod\limits_{k = 1}^{d-1}\Abs{(\Slim{C\in M_{2\times 2}(F)} \psi_1(Tr(-AC-AC^2))}\\
&\geq \frac{1}{|F|^{2d-2}}\prod\limits_{k = 1}^{d-1}\left(|F|^2|F|^\frac{1}{2}\right)\\
&= |F|^\frac{d-1}{2},
\end{split}
\]
which grows as needed, where we used the computations from Lemma \ref{lowDimensionalExceptions} to go from the second line to the third line.\\

We may also tackle the case of $n = 3$ similarly, simplifying notation as here we have $d = 3$ by lemma \ref{mostCases}. Let $A$ be as in the proof of lemma \ref{lowDimensionalExceptions}. We compute:
\[
\begin{split}
\frac{|R|^3}{|V_{p, 0}|^\frac{1}{2}}\Abs{\widehat{V_{p, 0}}(A, A, A)} &= \frac{1}{|F|^9}\Abs{\Slim{C_1, C_2\in M_{3\times 3}(F)} \psi_1(Tr(-AC_1-AC_2-AC_1^2-AC_2^2))} \\
&= \frac{1}{|F|^{9}}\prod\limits_{k = 1}^{2}\Abs{\Slim{C\in M_{3\times 3}(F)} \psi_1(Tr(-AC-AC^2))}\\
&\geq \frac{1}{|F|^{9}}\prod\limits_{k = 1}^{2}\left(|F|^6|F|^\frac{1}{2}\right)\\
&= |F|^4,
\end{split}
\] 
giving us the needed growth.
\end{proof}
Putting this all together, we may state the following theorem:
\begin{theorem}\label{simpleRings}
    Let $R$ be a finite simple ring (so it is the ring of $n\times n$ matrices over a finite field $F$). Suppose that the $d$-dimensional paraboloid $V_{p, c}$ is a $C$-Salem set. Then $R$ is either a finite field or comes from a finite list of exceptional rings.
\end{theorem}
\begin{proof}
    For $d\geq 4,$ $n>2$, $d = 3,$ $n>3,$ and $d = 2,$ $n>4$ this is lemma \ref{mostCases}. For $d = 2,$ $n = 2, 3, \text{or } 4,$ this is lemma \ref{lowDimensionalExceptions}. Finally, for $d = 3,$ $n = 2\text{ or } 3$ and for $d > 3$, $n = 2$ this is lemma \ref{higherDimensionalExceptions}.
\end{proof}

 \section{Acknowledgments}
The author wishes to express gratitude to Prof. Alex Iosevich and to Prof. Azita Mayeli for introducing him to this problem and for their support in completing this paper. The author completed this work while partially supported by a Graduate Center Fellowship at the CUNY Graduate Center.


\begin{thebibliography}{26}
\bibitem{AdolphSperb}  A. Adolphson, S. Sperber, {\it Exponential sums and Newton polyhedra}, Bull. Amer. Math. Soc. (N.S.) 16 (1987), no. 2, 282–286. 

\bibitem{Apostol} T. Apostol, {\it Introduction to Analytic Number Theory}, Undergraduate Texts in Mathematics. Springer-Verlag, New York, 1976. doi: https://doi.org/10.1007/978-1-4757-5579-4

\bibitem{Babai} L. Babai, {\it The Fourier Transform and Equations over Finite Abelian Groups}, University of Chicago class notes.

\bibitem{Bluhm} C. Bluhm, {\it Random Recursive Construction of Salem Sets}, Ark. Mat. 34 (1996), 51–63.

\bibitem{Chen} C. Chen. {\it Salem Sets in Vector Spaces over Finite Fields}, Ark. Mat., 56, (2018), 45–52.

\bibitem{Hamming} D. Cheong, D. Koh, and T. Pham, {\it Extension Theorems for Hamming Varieties over finite fields} Proc. Amer. Math. Soc. 150 (2022), no. 1, 161–170. doi: https://doi.org/10.1090/proc/15738

\bibitem{CovIoPak} D. Covert, A. Iosevich, and J. Pakianathan, {\it Geometric Configurations in the Ring of Integers Modulo $p^\ell$}, Indiana University Mathematics Journal Vol. 61, No. 5 (2012), 1949-1969. doi: https://doi.org/10.1512/iumj.2012.61.4751

\bibitem{Deligne1} P. Deligne, {\it La conjecture de Weil I}, Inst. Hautes \' Etudes Sci. Publ. Math. 43 (1974), 273–
307.

\bibitem{Deligne2} P. Deligne, {\it La conjecture de Weil : II,} Publications Math\' ematiques de l'IH\' ES, Volume 52 (1980), pp. 137-252. 

\bibitem{DeligneSGA} P. Deligne, {\it Applications de la formule des traces aux sommes trigonom\' etriques}, in [SGA4$\frac{1}{2}$], Springer Lectures Notes 569, PP. 168-232, Springer-Verlag, Berlin-Heidelberg-New York, 1977 doi: https://doi.org/10.1007/BFb0091516

\bibitem{Fraser} J. Fraser, {\it $L^p$ Averages of the Fourier Transform in Finite Fields}, (https://arxiv.org/abs/2407.08589), preprint.

\bibitem{FalconerConj} D. Hart, A. Iosevich, D. Koh, M. Rudnev, {\it Averages over Hyperplanes, Sum-Product Theory in Vector Spaces over Finite Fields and the Erd\H os-Falconer Distance Conjecture}, Transactions of the American Mathematical Society, Volume 363, Number 6, June 2011, 3255-3275 doi: https://doi.org/10.1090/S0002-9947-2010-05232-8

\bibitem{IoMorgPak} A. Iosevich, H. Morgan, and J. Pakianathan, {\it On Directions Determined by
Subsets of Vector Spaces over Finite Fields}, Integers 11 (2011), 815–825.

\bibitem{IoMPak} A. Iosevich, B. Murphy, and J. Pakianathan, {\it The Square Root Law and Structure of Finite Rings}, Mosc. J. Comb. Number Theory 7 (2017), no. 1, 38-72

\bibitem{IoRud} A. Iosevich, M. Rudnev, {\it Erd\H os Distance Problem in Vector Spaces over Finite Fields}, Transactions of the American Mathematical Society, Volume 359, Number 12, Dec. 2007, 6127–6142

\bibitem{Katz} N. Katz, {\it Estimates for ``Singular'' Exponential Sums}, International Mathematics Research Notices, Volume 1999, Issue 16, 1999, 875–899. doi: https://doi.org/10.1155/S1073792899000458

\bibitem{KiehlWeissauer} R. Kiehl and R. Weissauer, {\it Weil Conjectures, Perverse Sheaves and l'adic Fourier Transform}, Ergebnisse der Mathematik und ihrer Grenzgebiete {\bf 42}. Springer-Verlag, Berlin-Heidelberg-New York, 2001. doi: https://doi.org/10.1007/978-3-662-04576-3

\bibitem{KohShen} D. Koh and C-Y. Shen, {\it Additive Energy and the Falconer Distance Problem
in Finite Fields}, Integers 13 (2013), 1–10.

\bibitem{Lang} S. Lang, {\it Algebra}, Graduate Texts in Mathematics {\bf 211}. Springer-Verlag, New York, revised third edition, 2002. doi: https://doi.org/10.1007/978-1-4613-0041-0

\bibitem{LidlNied} R. Lidl and H. Niederreiter, {\it Finite Fields}, Encyclopedia of Mathematics and its Applications {\bf 20}. Cambridge Univ. Press, 1997. doi: https://doi.org/10.1017/CBO9780511525926

\bibitem{Mattila} P. Mattila, {\it Fourier Analysis and Hausdorff Dimension}, Cambridge Studies in Advanced Mathematics, vol. 150, Cambridge University Press, 2015.

\bibitem{Mazur} B. Mazur, {\it Finding Meaning in Error Terms}, Bull. Amer. Math. Soc., 45(2):185-228, 2008. doi: https://doi.org/10.1090/S0273-0979-08-01207-X

\bibitem{Roth1} K. Roth, {\it Sur quelques ensembles d'entiers}, Comptes Rendus, Volume 234, 1952, 388-390.

\bibitem{Roth2} K. Roth, {\it On Certain Sets of Integers},  Journal of the London Mathematical Society, Volumes 28, Issue 1, 1953, 104-109. doi:  https://doi.org/10.1112/jlms/s1-28.1.104

\bibitem{Rotman} J. Rotman, {\it Advanced Modern Algebra Part II}, Graduate Studies in Mathematics {\bf 180}. American Mathematical Society, third edition, 2017.

\bibitem{Tao} T. Tao, {\it The sum-product phenomenon in arbitrary rings}, Contributions to Discrete Mathematics, Volume 4, Number 2, Dec. 2009, 59-82. doi: https://doi.org/10.11575/cdm.v4i2.61994 

\end{thebibliography}
\end{document}